\documentclass[11pt, reqno]{amsart}
\usepackage[utf8]{inputenc}
\numberwithin{equation}{section}
\usepackage{color}
\usepackage[shortlabels]{enumitem}
\usepackage[hidelinks]{hyperref}
\usepackage{cleveref}
\usepackage{todonotes}
\newcommand{\qtq}[1]{\quad\text{#1}\quad}
\usepackage{mathabx}
\usepackage{amsmath, amssymb, amsfonts, amsthm, mathtools}
\usepackage{dsfont}

\usepackage[scr=boondoxo]{mathalfa}

\theoremstyle{definition}
\newtheorem{definition}{Definition}

\theoremstyle{plain}
\newtheorem{theorem}[definition]{Theorem}
\newtheorem*{theorem*}{Theorem}
\newtheorem{lemma}{Lemma}
\newtheorem{proposition}[definition]{Proposition}

\newtheorem{corollary}[definition]{Corollary}

\theoremstyle{remark}
\newtheorem*{claim*}{Claim}

\newtheorem{remark}[]{Remark}
\newtheorem*{remark*}{Remark}

\newtheorem*{example*}{Example}

\newcommand{\eps}{\varepsilon}

\DeclareMathOperator{\R}{\mathbb{R}}

\DeclareMathOperator{\bigO}{\mathcal{O}}
\DeclareMathOperator{\F}{\mathcal{F}}
\DeclareMathOperator{\K}{\mathcal{K}}
\DeclareMathOperator{\M}{\mathcal{M}}
\DeclareMathOperator{\D}{\mathcal{D}}

\newcommand{\bigo}{\mathcal{O}}

\newcommand{\jbrak}[1]{\langle#1\rangle}

\renewcommand{\d}{\mathrm{d}}
\renewcommand{\:}{\colon}

\newcommand{\bbar}{\overline}

\renewcommand{\hat}{\widehat}

\renewcommand{\d}{\mathrm{d}}

\newcommand{\dx}{\, \mathrm{d}x}
\newcommand{\dr}{\, \d r}

\newcommand{\ds}{\, \mathrm{d}s}

\newcommand{\cD}{\mathcal{D}}
\newcommand{\cE}{\mathcal{E}}

\newcommand{\cX}{\mathscr{X}}
\newcommand{\cW}{\mathcal{W}}
\newcommand{\cN}{\mathcal{N}}

\begin{document}
    \title[Modified Scattering]{Modified Scattering for the Schr\"odinger--Bopp--Podolsky Equation}
    \author[T. Van Hoose]{Tim Van Hoose}
    \address{Department of Mathematics, University of North Carolina, Chapel Hill}
    \email{tvh@unc.edu}
    \begin{abstract}
        We prove sharp $L^\infty$ decay and modified scattering for the Schr\"odinger--Bopp--Podolsky equation in $2$ and $3$ spatial dimensions with small initial data chosen from a weighted Sobolev space. 
    \end{abstract}
    \maketitle
    
    \section{Introduction}
    We will consider long time behavior of small solutions to the so-called Schr\"odinger--Bopp--Podolsky model:
    \begin{equation}\label{E:SBP}
        \begin{cases}
            i\partial_t u + \Delta u = (\K \ast |u|^2)u - |u|^{\frac{2}{d}}u \\
            u|_{t = 0} = u_0(x) 
        \end{cases}\text{on } \R \times \R^d, \ d \in \{2, 3\}
    \end{equation}
    where the convolution kernel $\K$ is the Bopp--Podolsky potential:
    \begin{equation}
        \K(x) = \frac{1-e^{-|x|}}{|x|}
    \end{equation}
    Our aim is to establish a modified scattering result for this model using techniques from \cite{hayashiAsymptoticsLargeTime1998a} in the weighted Sobolev space $H^{\gamma_d, \gamma_d}(\R^d)$, defined by the norm 
    $\|u\|_{H^{\gamma_d, \gamma_d}} = \|\jbrak{\nabla}^{\gamma_d} u\|_{L^2} + \| |x|^{\gamma_d} u\|_{L^2}$. 
    Our main result is the following:
    \begin{theorem}\label{T:MainTheorem}
        Let $d\in \{2, 3\}$ and $\gamma_d \in (\frac{d}{2}, 1+\frac{2}{d})$ defined by \eqref{E:gammad_defn}, and let $u_0 \in H^{\gamma_d, \gamma_d}(\R^d)$ satisfy $\|u_0\|_{H^{\gamma_d, \gamma_d}(\R^d)} = \eps > 0$. For $\eps$ sufficiently small, there exists a unique solution $u(t, x) \in C_t^0 H_x^{\gamma_d, \gamma_d}([0, \infty) \times \R^d)$ to \eqref{E:SBP} with $u(0, x) = u_0(x)$. Furthermore, the solution obeys 
        \begin{equation}
            \|u(t)\|_{L^\infty} \lesssim \eps(1+|t|)^{-\frac{d}{2}}
        \end{equation}
        for all $t \geq 0$. There also exists $\cW \in L^\infty$ so that
        \begin{equation}
            \begin{multlined}
                u(t, x) = (2it)^{-\frac{d}{2}} e^{i|x|^2/4t}\left[\exp\left\{-\frac{i}{2}(\K \ast |\cW|^2 - |\cW|^{\frac{2}{d}})\left(\frac{x}{2t}\right)\right\}\cW\left(\frac{x}{2t}\right)\right] \\+ \mathcal{O}(t^{-\frac{d}{2}-\frac{1}{50} + (1+\frac{2}{d})\eps^2})
            \end{multlined}    
        \end{equation} 
    in $L^\infty(\R^d)$ as $t \to \infty$. 
    \end{theorem}
    The model \eqref{E:SBP} comes from the model 
    \begin{equation*}
        \begin{cases}
            i\partial_t u + \Delta u = V(x)u - |u|^{\frac{2}{d}} u \\
            -\Delta V + \Delta^2 V = |u|^2 \\
            u|_{t=0}=u_0(x)
        \end{cases}\text{on } \R \times \R^3
    \end{equation*} 
    Indeed, the equation for $V$ can be solved directly by taking the convolution of $|u|^2$ with the Green's function for $-\Delta + \Delta^2$. In three spatial dimensions one may verify our claim using the Fourier transform and some contour integration. \par
    
    The Bopp--Podolsky model of electrodynamics was introduced by Bopp and Podolsky (\cite{boppLineareTheorieElektrons1940b}, \cite{podolskyGeneralizedElectrodynamicsPart1942}) independently as a way to remove certain nonphysical divergences arising in classical electrodynamics. In short, classical electrodynamics suggests that the electrostatic potential for a charge should be governed by Poisson's equation $-\Delta \phi = \rho$. For a point charge (for example) $\rho = \delta_0$ and the fundamental solution to this equation is $\phi(x) \sim |x|^{-1}$. The energy for the particle is (proportional to) the $\dot{H}^1(\R^3)$-norm of $\phi$, which is clearly infinite in this case. \par
    However, upon replacing the Poisson equation $-\Delta\phi = \rho$ with the Bopp--Podolsky equation $-\Delta\phi +\Delta^2 \phi = \rho$ (and modifying the energy in the natural way), one can check that the energy of the point charge is indeed finite. \par
    For other information about the Bopp-Podolsky model, see \cite{daveniaNonlinearSchrodingerEquation2019,gaoNonlinearSchrodingerEquation2023,heNormalizedSolutionsSchrodingerBoppPodolsky2022,zhengExistenceFiniteTime2022}. 
    Our goal in this work is to prove a global existence and modified scattering result for \eqref{E:SBP} for small initial data chosen from a suitable weighted Sobolev space. We will heavily rely on the classical approach of \cite{hayashiAsymptoticsLargeTime1998a}, as the nonlinearity in our equation fits into their framework very nicely. Indeed, our potential $\K(x)$ is $\bigo(|x|^{-1})$ as $x \to \infty$ and we chose the scattering-critical power $2/d$ for the power-type nonlinearity. Such modified scattering results have been studied in a variety of contexts and methods. We direct the interested reader to \cite{murphyReviewModifiedScattering2021} for a high-quality exposition of the differing methods, as well as \cite{ifrimTestingWavePackets2022,katoNewProofLong2010a} for two specific cases. Additionally, the classical papers \cite{ginibreLongRangeScattering1993,ginibreExistenceWaveOperators1994,hirataCauchyProblemHartree1991,hayashiScatteringTheoryHartree1998} (and others; see references) may lend some alternative viewpoints into our methods, as well as some background into the study of Schr\"odinger equations with nonlocal nonlinearities. \par
    Our method is based on the one employed in \cite{hayashiAsymptoticsLargeTime1998a}. In their paper, the proof of the global existence and sharp $L^\infty$ decay result arises as the combination of estimates of an `energy' norm and a `dispersive' norm. To estimate the energy norm, one uses some chain-rule type estimates and Gr\"onwall's inequality, while the dispersive norm is estimated by writing down an approximate ODE for the `profile' $\hat{f}(t) = \F e^{-it\Delta}u(t)$. We will employ essentially the same analysis; the only subtleties come from the inhomogeneity of the potential $\K(x)$ as compared to the Hartree potential of \cite{hayashiAsymptoticsLargeTime1998a}, which is indeed homogeneous of degree -1. For more details, see \Cref{P:higherdimDnorm}. \par
        
    The remainder of the paper is organized as follows: In \Cref{S:Notation}, we gather some notation along with useful identities and inequalities. In \Cref{SS:LocalTheory}, we prove the local well-posedness result for \eqref{E:SBP} in a suitable weighted Sobolev space. In \Cref{SS:GlobalExistence}, we prove the first part of the main theorem regarding the global existence and sharp $L^\infty$ decay. Finally, in \Cref{S:AsympBehavior}, we establish the claimed asymptotic behavior of solutions.

    \subsection*{Acknowledgements} We are grateful to Jeremy Marzuola for many fruitful discussions, as well as to Jason Murphy for bringing this problem to our attention and for helpful comments regarding the exposition. T.V.H. was supported by the UNC NSF DMS-2135998 Research Training Grant.
    
    \section{Notation and Preliminaries}\label{S:Notation}
        In this section, we introduce some notation that will be used throughout the remainder of the paper.\par
        We write $A \lesssim B$ or $B \gtrsim A$ to denote the inequality $A \leq CB$ for some constant $C > 0$, where $C$ may depend on parameters like the dimension or the indices of function spaces. If $A \lesssim B $ and $B \lesssim A$ both hold, then we write $A \sim B$. We will also make use of the standard Landau symbol $\bigO$, as well as the Japanese bracket notation $\jbrak{\cdot} := (1+|\cdot|^2)^{\frac{1}{2}}$. \par
        We use the standard Lorentz spaces $L^{p,q}(\R^d)$, specifically the fact that $|x|^{-1} \in L^{d, \infty}(\R^d)$. \par
        For our local well-posedness result (see \Cref{SS:LocalTheory}), we need the standard (endpoint) Strichartz estimates for the Schr\"odinger equation proven in \cite{taoNonlinearDispersiveEquations2006c} and \cite{keelEndpointStrichartzEstimates1998}.
        We will denote the Fourier transform by $\F[f](\xi) = \widehat{f}(\xi)$ with the normalization 
        \begin{equation*}
            \F[f](\xi) := (2\pi)^{-\frac{d}{2}} \int_{\R^d} e^{-ix \cdot\xi} f(x) \dx 
        \end{equation*}
        and inverse 
        \begin{equation}
            \F^{-1}[g](x) = \check{g}(x) := (2\pi)^{-\frac{d}{2}} \int_{\R^d} e^{ix\cdot\xi} g(\xi)\d \xi
        \end{equation}
        We will define the weighted Sobolev spaces $H^{\gamma, \nu}(\R^d)$ by the norm
        \begin{equation}
            \|u\|_{H^{\gamma, \nu}} := \|\jbrak{\nabla}^\gamma u\|_{L^2} + \| |x|^\nu u\|_{L^2}
        \end{equation}
        where as usual $\jbrak{\nabla}^\gamma := \F^{-1} \jbrak{\xi}^\gamma \F$, and we define the standard Sobolev spaces $H^s(\R^d):= H^{s, 0}(\R^d)$ in the notation above. \par
        In the usual way, we will denote the free Schr\"odinger propagator by $e^{it\Delta}$. Direct computation shows that we can decompose this operator as 
        \begin{equation}\label{E:MDFM}
            e^{it\Delta} = \M(t) \cD(t) \F \M(t)
        \end{equation}
        where 
        \begin{equation}
            \M(t)f(x) = e^{i\frac{|x|^2}{4t}} f(x) \qquad \text{and} \qquad \cD(t) = (2it)^{-\frac{d}{2}} f\left(\frac{x}{2t}\right).
        \end{equation}
        By direct computation, we see that 
        \begin{equation}
            \cD(t)^{-1} = (2i)^{d} \cD\left(\frac{1}{t}\right) 
        \end{equation}
        We will also make use of the Galilean operator $J(t):= x + 2it\nabla$. Direct computation shows that
        \begin{equation}\label{E:Jtdefn1}
            J(t) = \M(t) (2it\nabla) \M(-t).
        \end{equation}
        Indeed, we have 
        \begin{align*}
            \M(t)(2it\nabla)\M(-t) f &= e^{i\frac{|x|^2}{4t}}(2it\nabla)e^{-i\frac{|x|^2}{4t}}f \\
            &= e^{i\frac{|x|^2}{4t}}e^{-i\frac{|x|^2}{4t}} (x+2it\nabla)f \\
            &= J(t)f.
        \end{align*}
        An ODE argument furnishes the identity 
        \begin{equation}\label{E:Jtdefn2}
            J(t) = e^{it\Delta}xe^{-it\Delta}
        \end{equation}
        Indeed, both sides of the equation match at $t=0$. If we take a time derivative on both sides, we find that 
        \begin{align*}
            J'(t) &= 2i\nabla \\
            \frac{d}{dt} [e^{it\Delta} x e^{-it\Delta}] &= i e^{it\Delta} [\Delta, x] e^{-it\Delta}
        \end{align*}
        where $[\cdot, \cdot]$ is the usual operator commutator. One can check directly that $[\Delta, x] = 2\nabla$. Since Fourier multipliers commute and $(e^{it\Delta})_{t \in \R}$ is a semigroup, we see 
        \begin{equation*}
            i e^{it\Delta} [\Delta, x] e^{-it\Delta} = 2i\nabla
        \end{equation*} 
        Since the time derivatives match for all $t$ and the two expressions have matching values at $t=0$, ODE uniqueness implies that $J(t) = e^{it\Delta}x e^{-it\Delta}$ for all $t$, which was what we claimed. 
        \begin{remark*}
            Since $e^{it\Delta}: L^2 \to L^2$ is unitary, we see that 
            \begin{equation*}
                \| J(t)u\|_{L^2} = \|x e^{-it\Delta}u\|_{L^2}.
            \end{equation*}
        \end{remark*}
        We also define powers of $J(t)$ in the following manner:
        \begin{align}
            |J|^{\gamma}(t) &:= \M(t)(-4t^2 \Delta)^{\frac{\gamma}{2}}\M(-t) \quad \text{for }\gamma\in[0, \infty) \label{E:Jtpower}\\
            &= e^{it\Delta}|x|^\gamma e^{-it\Delta} \label{E:Jtpower2}
        \end{align}
        Next, we record some of the regularity properties of $\K(x)$ that will be of use to us later.
        \begin{lemma}\label{L:Kregularity}
            Let $\K(x) : \R^d \to \R$ be defined as above. Then 
            \begin{enumerate}
                \item $\K \in L^{d, \infty}(\R^d)$ and $\K \in L^p(\R^d)$ for all $d < p \leq \infty$. 
                \item The following convolution estimate for $\K$ holds:
                \begin{equation*}
                    \|\K \ast f\|_{L^r} \lesssim \|f\|_{L^q}
                \end{equation*}
                where $q < r$ and $1-\tfrac{1}{d} \leq \tfrac{1}{q}-\tfrac{1}{r} \leq 1$.
            \end{enumerate}
        \end{lemma}
        \begin{proof}
            To see the first part of (1), note that we have $\K(x) \leq |x|^{-1}$ pointwise in $x$. This readily implies that $\K(x)$ is in $L^{d, \infty}(\R^d)$ because $x \mapsto |x|^{-1} \in L^{d, \infty}(\R^d)$. For the second part of the (1), we directly integrate:
            \begin{align*}
                \|\K(x)\|_{L^p}^p &= \int_{\R^d} \left(\frac{1-e^{-|x|}}{|x|}\right)^p \dx \\
                &\leq \int_{B(0, 1)} \left(\frac{1-e^{-|x|}}{|x|}\right)^p \dx + \int_{\R^d \setminus B(0, 1)} \left(\frac{1-e^{-|x|}}{|x|}\right)^p \dx \\
                &\lesssim 1 + \int_{\R^d \setminus B(0, 1)} \left(\frac{1-e^{-|x|}}{|x|}\right)^p \dx\\
                &\lesssim 1+ \int_{\R^d \setminus B(0, 1)}  |x|^{-p} \dx \\
                &\lesssim 1+ \int_{S^{d-1}}\int_1^\infty r^{d-1-p} \dr \d \Omega,
            \end{align*}    
            where from the second to the third line we used that $\K \in L^\infty$, and from the third to the fourth line we used that $\K(x) \leq |x|^{-1}$. In the final line we switched to spherical coordinates. In particular, the last integral in spherical coordinates converges if and only if $d - p < 0$, i.e. if $p > d$. \par
            To prove (2), note that by Young's convolution inequality 
            \begin{equation}\label{E:youngs}
                \|\K \ast f\|_{L^r} \lesssim \|\K\|_{L^p}\|f\|_{L^q}
            \end{equation}
            where $1+\tfrac{1}{r} = \tfrac{1}{p}+ \tfrac{1}{q}$ for any $d < p \leq \infty$. Rearranging the former equality gives $\tfrac{1}{q}-\tfrac{1}{r} = 1-\tfrac{1}{p}$. The conditions on $p$ then guarantee that $\tfrac{1}{q}-\tfrac{1}{r} \in (1-\tfrac{1}{d}, 1]$. To conclude, we note that a stronger version of Young's inequality holds, namely that we can replace the $L^p$-norm in \eqref{E:youngs} by $L^{p, \infty}$. This allows us to include $d$ in the range of $p$ for which the inequality holds, and thus include the endpoint $1-\tfrac{1}{d}$. 
        \end{proof}
        
        Another trivial estimate we will make use of in the sequel is the following consequence of the mean value theorem: 
        \begin{lemma}\label{L:mvtbound}
            For any $x \in \R \text{ and } 0\leq \alpha \leq 1$,
            \[|e^{ix}-1| \lesssim |x|^{\alpha}.\]
        \end{lemma}
        For notational convenience, for the remainder of the paper we define
        \begin{equation}\label{E:gammad_defn}
            \gamma_d = \tfrac{1}{2}+\tfrac{d^2+4}{4d} \in (\tfrac{d}2, 1+\tfrac2d),   
        \end{equation}
        noting that this interval is indeed nonempty for $d \in \{2,3\}$. If $X([0, T];Y)$ is a Bochner space of functions $[0, T]\to Y$, we will write $X_T = X([0, T])$. 
        
    \section{Proof of Main Theorem}\label{S:MTProof}
        \subsection{Local Theory}\label{SS:LocalTheory}    
        To get the bootstrap argument started, we need a local existence theory for \eqref{E:SBP}. Commuting derivatives with the nonlinearity can be handled using the algebra property of $H^{\gamma_d}$. To deal with commuting spatial weights with the nonlinearity, we use the operator $|J|^{\gamma_d}(t)$ defined by \eqref{E:Jtpower}. A similar result in this spirit can be found in \cite[Proposition 2.1]{ifrimGlobalBoundsCubic2014}.
        \begin{theorem}\label{T:localexistence}
            For any $u_0 \in H^{\gamma_d, \gamma_d}(\R^d)$, there exists $T >0$ and a unique solution $u \in C_T^0 H_x^{\gamma_d, \gamma_d}$ with $u(0, x) = u_0(x)$. Further, we have the standard blowup alternative: either $u$ is forward-global, or there exists $T_\ast >0$ such that 
            \[
              \limsup_{t \to T_\ast^{-}} \|u(t)\|_{H^{\gamma_d}} = \infty.  
            \]
        \end{theorem}
        \begin{proof}
            To obtain existence in $H^{\gamma_d}(\R^d)$, we proceed as discussed above. Our argument is a standard contraction mapping argument in the space 
            \[Y^{\gamma_d} = \{u \in L_t^\infty H_x^{\gamma_d}([0, T] \times \R^d) \mid \|u\|_{L_T^\infty H_x^{\gamma_d}} \leq 2C\|u_0\|_{H_x^{\gamma_d}}\},\] 
            where the factor $C$ encodes the constants in the inequalities we use. We clearly have that $Y^{\gamma_d}$ is a closed (hence complete) subspace of $L_T^\infty H_x^{\gamma_d}$. To begin, recast \eqref{E:SBP} as a Duhamel integral operator: 
            \begin{align}
                \Phi[u] &= e^{it\Delta}u_0 - i \int_0^t e^{i(t-s)\Delta} \left[(\K \ast |u|^2)u(s) - |u|^{\frac{2}{d}}u\right] \ds \label{E:DUHAMEL}\\
                        &:= e^{it\Delta}u_0 -i \int_0^t e^{i(t-s)\Delta} (\cN_1(u)(s) - \cN_2(u)(s)) \ds 
            \end{align}
            where we defined $\cN_1(u) = (\K \ast |u|^2)u$ and $\cN_2(u) = |u|^{\frac{2}{d}}u$. 
            
             For any $u \in Y^{\gamma_d}$, we then have (by Strichartz) 
             \begin{equation}\label{E:Phiest}
                \begin{multlined}[t]
                \|\jbrak{\nabla}^{\gamma_2} \Phi[u]\|_{L_T^\infty H_x^{\gamma_d}} \lesssim \|u_0\|_{H^{\gamma_d}} 
                      \\+\left\| \jbrak{\nabla}^{\gamma_d} \int_0^t e^{i(t-s)\Delta} (\cN_1(u)(s) - \cN_2(u)(s)) \ds\right\|_{L_T^\infty L_x^2}.                    
                \end{multlined}
            \end{equation}
            To estimate the integral term, we need to split into the $2d$ and $3d$ cases separately. \par
            In $2$ dimensions, we have that $L_{t, x}^{\frac{4}{3}}$ and $L_t^1 L_x^2$ are both dual admissible Strichartz spaces. This gives the estimate 
            \begin{multline*}
                \left\| \jbrak{\nabla}^{\gamma_d} \int_0^t e^{i(t-s)\Delta} (\cN_1(u)(s) - \cN_2(u)(s)) \ds\right\|_{L_T^\infty L_x^2} \lesssim \| \jbrak{\nabla}^{\gamma_2} \cN_1(u)\|_{L_{T,x}^{\frac{4}{3}}} \\+ \| \jbrak{\nabla}^{\gamma_2} \cN_2(u)\|_{L_T^1 L_x^2}.
            \end{multline*}
            By the algebra property of $H^{\gamma_2}(\R^2)$ and H\"older in time, the latter term on the right-hand side is bounded by 
            \begin{equation}\label{E:N2est2d}
                \|\jbrak{\nabla}^{\gamma_2} \cN_2(u)\|_{L_T^1 L_x^2} \lesssim T\|u\|_{L_T^\infty H_x^{\gamma_2}}^2,
            \end{equation}
            which is acceptable (note that in this case we used that $1+\tfrac{2}{d} = 2$ when $d =2$). For $\cN_1$, we use the fractional product rule and commutativity of Fourier multipliers to bound 
            \begin{align*}
                \| \jbrak{\nabla}^{\gamma_2}\cN_1(u)\|_{L_{T,x}^{\frac{4}{3}}} &\lesssim \| (\K \ast |u|^2)u\|_{L_{T,x}^{\frac{4}{3}}} +\| \K \ast |u|^2\|_{L_T^\infty L_x^4} \| |\nabla|^{\gamma_2}u\|_{L_T^{\frac{4}{3}}L_x^2} \\ &\quad+ \| \K \ast (|\nabla|^{\gamma_2} |u|^2)\|_{L_T^\infty L_x^4} \|u\|_{L_T^{\frac{4}{3}}L_x^2}.
            \end{align*}
            The terms featuring copies of only $u$ or derivatives of $u$ are controlled by $T^{\frac{3}{4}} \|u\|_{L_T^\infty H_x^{\gamma_d}}$ by applying H\"older's inequality in time and Sobolev embedding, which is acceptable. It remains to handle the terms containing convolutions, which we do with the aid of the second part of \Cref{L:Kregularity}. To wit, by H\"older and \Cref{L:Kregularity},
            \begin{align*}
                \| (\K \ast |u|^2) u\|_{L_{T, x}^{\frac{4}{3}}} &\lesssim \| u\|_{L_T^{\frac{4}{3}}L_x^2} \| \K \ast |u|^2\|_{L_T^\infty L_x^4} \\
                &\lesssim T^{\frac{3}{4}}\|u\|_{L_T^\infty L_x^2} \| |u|^2\|_{L_T^\infty L_x^1} \\
                &\lesssim T^{\frac{3}{4}} \|u\|_{L_T^\infty H_x^{\gamma_2}}^3.          
            \end{align*}
            The above analysis also implies 
            \begin{align*}
                \|\K \ast |u|^2\|_{L_T^\infty L_x^4} \lesssim \|u\|_{L_T^\infty H_x^{\gamma_2}}^2,
            \end{align*}
            which is acceptable. For the term with derivatives, we have 
            \begin{align*}
                \| \K \ast |\nabla|^{\gamma_2}|u|^2\|_{L_T^\infty L_x^4} &\lesssim \| |\nabla|^{\gamma_2} |u|^2\|_{L_T^\infty L_x^{\frac{4}{3}}}\\
                &\lesssim \| u\|_{L_T^\infty \dot{H}_x^{\gamma_2}} \|u\|_{L_T^\infty L_x^4} \\
                &\lesssim \|u\|_{L_T^\infty H_x^{\gamma_2}}^2
            \end{align*}
            where the final inequality follows from inhomogeneous Sobolev embedding and the trivial embedding $H^s \hookrightarrow \dot{H}^s$. Combining all the estimates together, we see that 
            \begin{equation}\label{E:N1est2d}
                \| \jbrak{\nabla}^{\gamma_2}\cN_1(u)\|_{L_{T, x}^{\frac{3}{4}}} \lesssim T^{\frac{4}{3}} \|u\|_{L_T^\infty H_x^{\gamma_2}}^3,
            \end{equation}
            and plugging \eqref{E:N2est2d} and \eqref{E:N1est2d} into \eqref{E:Phiest}, a standard bootstrap argument in the space $Y^{\gamma_2}$ gives local existence of $H^{\gamma_d}$ solutions to \eqref{E:SBP} in $2d$. \par
            In $3d$, we proceed in much the same way as in $2$ dimensions. Indeed, starting with \eqref{E:Phiest}, we recognize $L_t^2 L_x^6$ and $L_t^\infty L_x^2$ as admissible pairs in $3d$ (the former space is an allowable endpoint; see \cite{keelEndpointStrichartzEstimates1998}), so we see that we need to estimate (by Strichartz)
            \[
              \| \jbrak{\nabla}^{\gamma_3} [(\K \ast |u|^2)u]\|_{L_T^2 L_x^{\frac{6}{5}}}  \qtq{and} \| \jbrak{\nabla}^{\gamma_3} |u|^{\frac{2}{3}}u\|_{L_T^1 L_x^2}.
            \]
            Again using the algebra property of $H^{\gamma_3}(\R^3)$, the second term can be bounded by 
            \begin{equation*}
                \| \jbrak{\nabla}^{\gamma_3} |u|^{\frac{2}{3}}u\|_{L_T^1 L_x^2}\lesssim T \|u\|_{L_T^\infty H_x^{\gamma_3}}^{\frac{5}{3}},
            \end{equation*}
            which is acceptable. To bound the contribution of the convolution nonlinearity, we use the fractional product rule to write 
            \begin{align*}
                \| \jbrak{\nabla}^{\gamma_3} [(\K \ast |u|^2)u]\|_{L_T^2 L_x^{\frac{6}{5}}} &\lesssim \| (\K \ast |u|^2)u\|_{L_T^2 L_x^{\frac{6}{5}}} + \| |\nabla|^{\gamma_3}u\|_{L_T^2 L_x^2} \| \K \ast |u|^2\|_{L_T^\infty L_x^3} \\ &\quad+ \| \K \ast |\nabla|^{\gamma_3} |u|^2\|_{L_T^\infty L_x^6} \|u\|_{L_T^2 L_x^{\frac{3}{2}}}.
            \end{align*}
            In this case, we need to Sobolev embed to control $\|u\|_{L_T^2 L_x^{\frac{3}{2}}} \lesssim T^{\frac{1}{2}}\|u\|_{L_T^\infty H_x^{\gamma_3}}$. The other term containing derivatives of $u$ can also be controlled by $T^{\frac{1}{2}} \|u\|_{L_T^\infty H_x^{\gamma_3}}$. The remaining terms that contain convolutions can be estimated by using \Cref{L:Kregularity}, Sobolev embedding and H\"older's inequality. Indeed, we have 
            \begin{align*}
                \|(\K \ast |u|^2)u\|_{L_T^2 L_x^{\frac{6}{5}}} &\lesssim \|u\|_{L_{T,x}^2} \| \K \ast |u|^2\|_{L_T^\infty L_x^3}\\
                &\lesssim T^{\frac{1}{2}}\|u\|_{L_T^\infty L_x^2}\| |u|^2\|_{L_T^\infty L_x^1} \\
                &\lesssim T^{\frac{1}{2}} \|u\|_{L_T^\infty L_x^2}^3\\
                &\lesssim T^{\frac{1}{2}}\|u\|_{L_T^\infty H_x^{\gamma_3}}^3
            \end{align*}
            The proof of this estimate directly implies the estimate 
            \begin{equation*}
                \| \K \ast |u|^2 \|_{L_T^\infty L_x^3} \lesssim \|u\|_{L_T^\infty H_x^{\gamma_3}}^2.
            \end{equation*}
            For the terms with derivatives, we have 
            \begin{align*}
                \| \K \ast |\nabla|^{\gamma_3} |u|^2\|_{L_T^\infty L_x^6} &\lesssim \| |\nabla|^{\gamma_3} |u|^2\|_{L_T^\infty L_x^{\frac{6}{5}}} \\
                &\lesssim \| |\nabla|^{\gamma_3} |u|^2\|_{L_T^\infty L_x^{\frac{6}{5}}} \\
                &\lesssim \| u\|_{L_T^\infty L_x^3}\||\nabla|^{\gamma_3}u\|_{L_T^\infty L_x^2}\\
                &\lesssim \|u\|_{L_T^\infty H_x^{\gamma_3}}^2.
            \end{align*}
            This tells us that $\|\jbrak{\nabla}^{\gamma_3}[(\K \ast |u|^2)u]\|_{L_T^2 L_x^{\frac{6}{5}}} \lesssim T^{\frac{1}{2}}\|u\|_{L_T^\infty H_x^{\gamma_d}}$. The same bootstrap argument as in the $2d$ case gives local existence in $H^{\gamma_3}(\R^3)$, as desired. \par
            Next, we want to upgrade from $H^{\gamma_d}(\R^d)$ initial data to data in the weighted Sobolev space $H^{\gamma_d, \gamma_d}(\R^d)$. To make this happen, we will use a persistence of regularity argument, which will show that for a solution to \eqref{E:SBP} with data in $H^{\gamma_d}$ on an interval $I = [0, T]$, the quantity 
            \[\| |J|^{\gamma_d}(t)u(t)\|_{L_t^\infty L_x^2 (I \times \R^d)}\] is bounded (and in particular cannot blow up). To this end, we'll use the representation \eqref{E:DUHAMEL} and the same Strichartz estimates from above. Assuming that our datum $u_0 \in H^{\gamma_d, \gamma_d}(\R^d)$ and using \eqref{E:Jtpower2} and Strichartz, we have 
            \begin{multline}\label{E:persistence}
                \| |J|^{\gamma_d}(t) u(t)\|_{L_t^\infty L_x^2(I\times \R^d)} \lesssim \| |J|^{\gamma_d} e^{it\Delta} u_0\|_{L_t^\infty L_x^2} \\+ \| |J|^{\gamma_d}(t) [(\K \ast |u|^2)u]\|_{A_d} + \| |J|^{\gamma_d}(t) |u|^{\frac{2}{d}}u\|_{L_T^1 L_x^2}.
            \end{multline}
            where $A_{d=2}$ = $L_{T, x}^{\frac{4}{3}}$ in dimension $2$ and $A_{d=3}$ = $L_T^2 L_x^{\frac{6}{5}}$ in dimension $3$. 
            By using the representation \eqref{E:Jtpower} and the fact that 
            \[
              e^{i\phi} \cN_j(u) = \cN_j(e^{i\phi}u)  \qtq{for} j = 1,2, \, \phi \in \R
            \]
            (i.e. that we have gauge invariant nonlinearities), we see that if we set $v = M(-t) u$, what we really have to understand is the quantities
            \begin{align}
                &\| (-4t^2 \Delta)^{\frac{\gamma_2}{2}} [(\K \ast |v|^2) v]\|_{L_{T, x}^{\frac{4}{3}}}, \label{E:T1}\\
                &\|(-4t^2 \Delta)^{\frac{\gamma_3}{2}} [(\K \ast |v|^2) v]\|_{L_T^2 L_x^{\frac{6}{5}}},\text{ and }  \label{E:T2}\\
                &\| (-4t^2\Delta)^{\frac{\gamma_d}{2}} |v|^{\frac{2}{d}}v\|_{L_T^1 L_x^2} \label{E:T3}.
            \end{align}
            However, the estimates of the derivatives of these quantities from earlier apply out of the box to show that each of them can be controlled in terms of $\|u_0\|_{H^{\gamma_d, \gamma_d}}^\theta$ for some power $\theta$. To be completely explicit, we have 
            \begin{equation}
                \eqref{E:T1} \lesssim T^{\frac{4}{3}}\|u\|_{L_T^\infty H_x^{\gamma_2}}^2 \||J|^{\gamma_2}u\|_{L_T^\infty L_x^2},
            \end{equation}
            \begin{equation}
                \eqref{E:T2} \lesssim T^{\frac{1}{2}} \|u\|_{L_T^\infty H_x^{\gamma_3}}^2 \| |J|^{\gamma_3}u\|_{L_T^\infty L_x^2},
            \end{equation}    
            and
            \begin{equation}
                \eqref{E:T3} \lesssim T \|u\|_{L_T^\infty H_x^{\gamma_d}}^\frac{2}{d}\| |J|^{\gamma_d}u\|_{L_T^\infty L_x^2}.
            \end{equation}
            Now if we insert these estimates back into \eqref{E:persistence} and choose $T$ small enough depending only on $\|u\|_{L_T^\infty H_x^{\gamma_d}}$, we see that we can guarantee 
            \begin{equation}
                \| |J|^{\gamma_d} u\|_{L_T^\infty L_x^2} \lesssim \| |x|^{\gamma_d} u_0\|_{L_x^2} + \frac{1}{2}\| |J|^{\gamma_d} u\|_{L_T^\infty L_x^2} 
            \end{equation}
            which implies that on the interval $[0, T]$ the quantity $\| |J|^{\gamma_d} u\|_{L_T^\infty L_x^2}$ at most doubles. To finish the argument, note that the $H^{\gamma_d}$ local well-posedness gives an interval of existence $I = [0, M]$ (where $M$ depends only on the size of $\|u_0\|_{H^{\gamma_d}}$). Dividing this interval into subintervals of length $T$ and applying the above argument, we see that $\| |J|^{\gamma_d}u\|_{L_T^\infty L_x^2}$ remains bounded on the lifespan of $u$, as desired.
            \par
            The last ingredient is the claimed blowup alternative. Here, we rely on \cite[Theorem 4.2]{katoNonlinearSchrodingerEquations1995}, which handles our exact situation. 
        \end{proof}

        \subsection{Global Existence and Decay}\label{SS:GlobalExistence}
        Let $u$ be a solution to \eqref{E:SBP} in $d$ spatial dimensions with $\|u_0\|_{H^{\gamma_d,\gamma_d}} = \eps$. By the local existence result above, we can assume that
            \begin{equation}\label{E:bootstraphyp1}
                \|u(t)\|_{H^{\gamma_d}} \lesssim 2\eps, \qtq{and} \| |J|^{\gamma_d}(t)u(t)\|_{L^2} \lesssim 2\eps     
            \end{equation} 
            for $t \in [0, 1]$ and that the solution exists on some maximal interval $(0, T_\ast)$ so that $T_\ast > 1$. \par
            The first part of our argument is a bootstrap for times $t \geq 1$ involving the following norms:
            \begin{equation}
                \begin{aligned}
                    &\|u(t)\|_{\cD} := \|\widehat{f}(t)\|_{L^\infty} \qtq{where} f(t):= e^{-it\Delta} u(t) \\
                    &\|u(t)\|_{\cE} := t^{-\eps^2}\{\|\jbrak{\nabla}^{\gamma_d}u(t)\|_{L^2} + \| |J|^{\gamma_d}(t)u(t)\|_{L^2}\}
                \end{aligned}
            \end{equation}

            Finally, define $\|u(t)\|_{\cX} = \sup_{s \in [1, T]} \{\|u(s)\|_{\cD} + \|u(s)\|_{\cE}\}$. 

            \begin{lemma}
                For any $t \geq 1$, 
                \begin{equation}\label{E:LinftytoX}
                    \|u(t)\|_{L^\infty} \lesssim t^{-\frac{d}{2}}\{\|u(t)\|_{\cD} + \|u(t)\|_{\cE}\}.
                \end{equation}
            \end{lemma}

            \begin{proof}
                Recall that $u(t) = \M(t) \D(t) \F \M(t) f(t)$, so that by Hausdorff-Young, \Cref{L:mvtbound}, Cauchy-Schwarz, and the fact that $\gamma_d > \tfrac{d}{2}$, we have
                \begin{align*}
                    \|u(t)\|_{L^\infty} &= \| \M(t)\D(t)\F \M(t) f(t)\|_{L^\infty} \\
                    &\lesssim t^{-\frac{d}{2}}\{\|\hat{f}(t)\|_{L^\infty} + \|\F \left[\M(t) - 1\right]f(t)\|_{L^\infty}\} \\
                    &\lesssim t^{-\frac{d}{2}}\{\|\hat{f}(t)\|_{L^\infty} + t^{-\eps^2} \| |x|^{2\eps^2} f(t)\|_{L^1}\}\\
                    &\lesssim t^{-\frac{d}{2}} \{\|\hat{f}(t)\|_{L^\infty} + t^{-\eps^2} \| \jbrak{x}^{\gamma_d}f\|_{L^2}\}\\
                    &\lesssim t^{-\frac{d}{2}} \{\|u\|_{\cD} + \|u\|_{\cE}\}
                \end{align*}
                which proves the claim.
            \end{proof}
            \begin{remark}
                Here and throughout, we tacitly assume that $\eps$ is smaller than some fixed $\eps_0$. In this lemma, for example, our choice of $\gamma_d$ tells us we need $\eps < \tfrac{\sqrt{5}}{2}$. In full generality (without fixing $\gamma_d$), there will be some dependence of $\eps$ on $\gamma_d$. 
            \end{remark}
            
            \begin{proposition}[Energy norm estimate]\label{P:higherdimEnorm}
                In dimensions $d = 2, 3$ and for any $t \geq 1$, the following estimate for the energy norm holds:
                \begin{equation}
                    \|u(t)\|_{\cE} \lesssim 2\eps t^{-\eps^2} + (1+\eps^{2-\frac{2}{d}})t^{-\eps^2} \int_1^t s^{-1+\eps^2} \|u\|_{\cX}^{1+\frac{2}{d}} \ds
                \end{equation}
            \end{proposition}
            \begin{proof}
                First, we note the trivial fact that $\K(x) \leq |x|^{-1}$. This allows us to bound \[|(\K \ast |u|^2)| \leq |(|\cdot|^{-1}\ast |u|^2)|.\] This replacement will allow us to take full advantage of the treatment of the $2$- and $3$-dimensional Hartree problem given in \cite{hayashiAsymptoticsLargeTime1998a}.\par
                To estimate the $\cE$-norm, we need to commute $|J|^{\gamma_d}$ with \eqref{E:SBP} and write the result as a Duhamel integral. This leads to 
                \begin{equation}
                    |J|^{\gamma_d} u(t) = e^{i(t-1)\Delta} |J|^{\gamma_d}(1)u(1) - i\int_1^t e^{i(t-s)\Delta}|J|^{\gamma_d}(s)[(\K \ast |u|^2)u(s) - |u|^{\frac{2}{d}}u(s)]\ds
                \end{equation}
                Taking the $L^2$-norm and using \eqref{E:Jtpower}, with $v = \M(-t)u$ we have 
                \begin{multline*}
                    \| |J|^{\gamma_d}u(t)\|_{L^2} \lesssim \| |J|^{\gamma_d}(1) u(1)\|_{L^2} + \int_1^t \| (-4t^2 \Delta)^{\frac{\gamma_d}{2}}[(\K \ast |v|^2)v]\|_{L^2} \\+ \| (-4t^2\Delta)^{\gamma_d} |v|^{\frac{2}{d}}v\|_{L^2}\ds.
                \end{multline*}
                The second term under the integral sign is controlled by 
                \[
                  \| (-4t^2\Delta)^{\frac{\gamma_d}{2}} |v|^{\frac{2}{d}}v\|_{L^2} \lesssim \| v\|_{L^\infty}^{\frac{2}{d}} \| |J|^{\gamma_d}u\|_{L^2},
                \]
                which is acceptable in light of the definition of $v$. For the first term, we note that we can modify the proof of \cite[Lemma 2.4]{hayashiAsymptoticsLargeTime1998a} to give the estimate 
                \[
                  \||J|^{\gamma_d} [(\K \ast |u|^2)u]\|_{L^2} \lesssim \|u\|_{L^\infty}^{\frac{2}{d}}\|e^{-it\Delta}u\|_{L^2}^{2-\frac{2}{d}}\||J|^{\gamma_d}u\|_{L^2}.
                \]
                Indeed, the details of the modification are as follows: we can replace the potential $V$ in the proof of Lemma 2.4 with $\K$ and change nothing. This nets us the estimate (in their notation)
                \begin{equation*}
                    |\Im\jbrak{|x|^{\gamma_d}e^{-it\Delta}(\K \ast |u|^2)u, |x|^{\gamma_d}e^{-it\Delta}u}| \lesssim \|g\|_{L^r}\| (-t^2\Delta)^{\frac{\gamma_d}{2}} f\|_{L^q} \|(-t^2\Delta)^{\frac{\gamma_d}{2}}g\|_{L^2}
                \end{equation*}
                where $\tfrac{1}{q} + \tfrac{1}{r}= \tfrac{1}{2}$. Then, since $f = \K \ast |g|^2 \leq |\cdot|^{-1} \ast |g|^2$ and Fourier multipliers commute, the same argument in their paper (i.e. replacing $|x|^{-1}$ with $|\nabla|^{-\frac{1}{2}(d-1)}$) yields
                \begin{equation*}
                    |\Im\jbrak{|x|^{\gamma_d}e^{-it\Delta}(\K \ast |u|^2)u, |x|^{\gamma_d}e^{-it\Delta}u}| \lesssim \|g\|_{L^r} \| (-t^2\Delta)^{\frac{\gamma_d}{2}}|g|^2\|_{L^{q_1}}\| |x|^{\gamma} e^{-it\Delta}u\|_{L^2}
                \end{equation*}
                where $\tfrac{1}{q} = \tfrac{1}{q_1} - \tfrac{d-1}{d}$. The remainder of the proof proceeds in the same way to give the claimed estimate. \par
                Putting everything together, we find 
                \begin{equation*}
                    \| |J|^{\gamma_d}u\|_{L^2} \lesssim 2\eps + \int_1^t \|u\|_{L^\infty}^{\frac{2}{d}}\||J|^{\gamma_d}u\|_{L^2} \left(1+\|e^{-it\Delta}u\|_{L^2}^{2-\frac{2}{d}}\right) \ds.
                \end{equation*}
                Recalling that $e^{-it\Delta}$ is unitary on $L^2$ and that mass is conserved by the Schr\"odinger flow, we conclude the desired estimate using the definition of the $\cX$-norm and the hypotheses on the inital data $u_0$. 
            \end{proof}
            \begin{proposition}[Dispersive norm estimate]\label{P:higherdimDnorm}
                In dimensions $d = 2, 3$ and for $t \geq 1$, the following estimate holds: 
                \begin{equation}
                    \|u(t)\|_{\cD} \lesssim 2\eps + \int_1^t s^{-1-\frac{1}{50}+3\eps^2} \|u\|_{\cX}^3 + s^{-1-\frac{1}{50}+(1+\frac{2}{d})\eps^2}\|u\|_{\cX}^{1+\frac{2}{d}} \ds
                \end{equation}
            \end{proposition}
            \begin{proof}
                Recall that we set $f(t) = e^{-it\Delta}u(t)$ for $u$ a solution to \eqref{E:SBP}; we will refer to $f$ as the `profile' of $u$. If we compute directly, we see that $f$ satisfies 
                \begin{equation}
                    i\partial_t f = e^{-it\Delta}\{(\K \ast |u|^2)u - |u|^{\frac{2}{d}}u\}.
                \end{equation}
                By \eqref{E:MDFM} and a bit of computation, we can rewrite the right-hand side in the form 
                \begin{multline}
                    i\partial_t f = (2t)^{-1} \M^{-1} \F^{-1}  (\K_t \ast |\cD^{-1} \M^{-1}u|^2)(\cD^{-1}\M^{-1}u) \\- (2t)^{-1} \M^{-1} \F^{-1}|\cD^{-1} \M^{-1} u|^{\frac{2}{d}} \cD^{-1} \M^{-1} u 
                \end{multline}
                where $\K_t(x)$ is defined by 
                \begin{equation}
                    \K_t(x) = \frac{1 - e^{-2t|x|}}{|x|}.
                \end{equation}
                Note that $\K_t(x) \leq |x|^{-1}$ for all $t > 0$. 
                We then use the fact that $u = \M \cD \F \M f$ to rewrite the right-hand side of the above one final time:
                \begin{multline}\label{E:profileeqn23d}
                    i\partial_t f = (2t)^{-1} \M^{-1} \F^{-1}  (\K_t \ast |\F \M f|^2) (\F\M f) \\- (2t)^{-1} \M^{-1} \F^{-1}|\F \M f|^{\frac{2}{d}} \F \M f .
                \end{multline}
                We will then perform one final algebraic manipulation, which basically involves adding and subtracting some counterterms. This leaves us with the equation
                    \begin{align*}
                    i\partial_t f - &(2t)^{-1} \F^{-1}\left[ (\K \ast |\hat{f}^2|)\hat{f} - |\hat{f}|^{\frac{2}{d}}\hat{f} \right] = \nonumber\\ 
                    &(2t)^{-1} (\M^{-1} - 1)\F^{-1}(\K_t \ast |\F \M f|^2) (\F\M f) \\ 
                    &+ (2t)^{-1} \F^{-1} \left\{(\K_t \ast |\F \M f|^2) (\F\M f) - (\K \ast |\hat{f}|^2)\hat{f}\right\} \\
                    &- (2t)^{-1} (\M^{-1}-1)\F^{-1}|\F\M f|^{\frac{2}{d}}\F \M f \\
                    &- (2t)^{-1} \F^{-1} \left\{|\F \M f|^{\frac{2}{d}}\F \M f - |\hat{f}|^{\frac{2}{d}} \hat{f}\right\} 
                    \end{align*}
                    
                Taking the Fourier transform, we have 
                \begin{align}
                    i\partial_t \hat{f} - &(2t)^{-1} \left[ (\K \ast |\hat{f}^2|)\hat{f} - |\hat{f}|^{\frac{2}{d}}\hat{f} \right] = \nonumber\\
                    &(2t)^{-1} \F (\M^{-1} - 1)\F^{-1}(\K_t \ast |\F \M f|^2) (\F\M f) \label{E:I1}\\ &+ (2t)^{-1} \left\{(\K_t \ast |\F \M f|^2) (\F\M f) - (\K \ast |\hat{f}|^2)\hat{f}\right\} \label{E:I2} \\
                    &- (2t)^{-1} \F(\M^{-1}-1)\F^{-1}|\F\M f|^{\frac{2}{d}}\F \M f \label{E:I3} \\
                    &- (2t)^{-1} \left\{|\F \M f|^{\frac{2}{d}}\F \M f - |\hat{f}|^{\frac{2}{d}} \hat{f}\right\} \label{E:I4}.
                \end{align}
                For notational brevity, rewrite the above display as 
                \begin{multline}
                    i\partial_t \hat{f} - (2t)^{-1} \left[ (\K \ast |\hat{f}^2|)\hat{f} - |\hat{f}|^{\frac{2}{d}}\hat{f} \right] =\eqref{E:I1} +\eqref{E:I2}+ \eqref{E:I3} + \eqref{E:I4} \\:=(2t)^{-1} \left\{I_1(t) + I_2(t) - I_3(t) - I_4(t)\right\},
                \end{multline}
                defining $I_j(t)$ in the natural way from the earlier equation and suggestive grouping of terms.
                
                We can remove the bracketed term on the left-hand side by introducing an integrating factor. Indeed, setting 
                \begin{equation}
                    B(t) = \exp\left(\frac{i}{2}\int_1^t \frac{1}{s} (\K \ast |\hat{f}|^2 - |\hat{f}|^{\frac{2}{d}}) \ds\right)
                \end{equation}
                and defining $g(t) = B(t) \hat{f}(t)$, we find the equation 
                \begin{equation}
                    i\partial_t g = (2t)^{-1}B(t)\{I_1 + I_2 - I_3 - I_4\}
                \end{equation}
                for $g(t)$. 
                By the fundamental theorem of calculus, 
                \begin{equation}
                    g(t) = g(1) - i \int_1^t (2s)^{-1} B(s)\{I_1(s) + I_2(s) - I_3(s)- I_4(s)\} \ds
                \end{equation}
                It then behooves us to provide a bound in $L^\infty$ for each $I_j(s)$. \par
                For $\|I_1(s)\|_{L^\infty}$, we have (by essentially the same argument as in \cite{hayashiAsymptoticsLargeTime1998a})
                \begin{align*}
                    \|I_1(s)\|_{L^\infty} &= \|\F (\M^{-1}-1)\F^{-1}(\K_t \ast |\F \M f|^2) (\F\M f)\|_{L^\infty} \\
                    &\leq \|(\M^{-1}-1)\F^{-1}\left[(\K_t \ast |\F\M f|^2)\F \M f\right]\|_{L^1}\\
                    &\lesssim t^{-\frac{1}{50}} \| |x|^{\frac{1}{25}} \F^{-1} \left[(\K_t \ast |\F \M f|^2)\F \M f\right] \|_{L^1} \\
                    &\lesssim t^{-\frac{1}{50}} \| \jbrak{x}^{\gamma_d} \F^{-1} \left[(\K_t \ast |\F \M f|^2)\F\M f\right]\|_{L^2} \\
                    &\lesssim t^{-\frac{1}{50}} \| \jbrak{\nabla}^{\gamma_d} \left[(\K_t \ast |\F \M f|^2) \F \M f\right]\|_{L^2}\\
                    &\lesssim t^{-\frac{1}{50}} \| \jbrak{\nabla}^{\gamma_d} (\K_t \ast |\F\M f|^2)\|_{L^2} \| \jbrak{\nabla}^{\gamma_d} \F \M f\|_{L^2}.
                \end{align*}
                To conclude, we estimate both quantities above: 
                \begin{align*}
                    \|\jbrak{\nabla}^{\gamma_d} (\K_t \ast |\F\M f|^2)\|_{L^2} &= \| \K_t \ast \jbrak{\nabla}^{\gamma_d} |\F\M f|^2\|_{L^2} \\
                    &\lesssim \| |\cdot|^{-1} \ast \jbrak{\nabla}^{\gamma_d} |\F \M f|^2\|_{L^2} \\
                    &\lesssim \| |\nabla|^{-\frac{1}{2}(d-1)} \jbrak{\nabla}^{\gamma_d} |\F \M f|^2\|_{L^2}\\
                    &\lesssim\begin{multlined}[t]
                        \| |\nabla|^{-\frac{1}{2}(d-1)} |\F \M f|^2\|_{L^2} \\+ \| |\nabla|^{-\frac{1}{2}(d-1)} |\nabla|^{\gamma_d} |\F \M f|^2\|_{L^2}
                    \end{multlined}
                \end{align*}
               We then see that by Plancherel, the first term can be controlled by (since $\gamma_d > d/2 \geq \frac{1}{2}$)
                \begin{align*}
                    \| |\nabla|^{-\frac{1}{2}(d-1)} |\F \M f|^2\|_{L^2} &\lesssim \| \jbrak{\xi}^{\gamma_d} |\xi|^{-\frac{1}{2}(d-1)} \F[|\F \M f|^2]\|_{L^2} \\
                    &\lesssim \| \jbrak{\nabla}^{\gamma_d} |\F \M f|^2\|_{L^2}\\
                    &\lesssim \| \jbrak{\nabla}^{\gamma_d} \F\M f\|_{L^2}^2
                \end{align*}
                and the second term can be controlled by 
                \begin{align*}
                    \| |\nabla|^{-\frac{1}{2}(d-1)} |\nabla|^{\gamma_d} |\F \M f|^2\|_{L^2} &\lesssim \| |\nabla|^{\gamma_d - \frac{d}{2}+ \frac{1}{2}} |\F \M f|^2\|_{L^2} \\
                    &\lesssim \| \jbrak{\nabla}^{\gamma_d} |\F \M f|^2\|_{L^2}\\
                    &\lesssim \| \jbrak{\nabla}^{\gamma_d} \F \M f\|_{L^2}^2.
                \end{align*}
                In particular, the bound on $\| \jbrak{\nabla}^{\gamma_d} \F \M f\|_{L^2}$ we need to provide will finish the job for the two terms above. To this end, we have 
                \begin{align*}
                    \| \jbrak{\nabla}^{\gamma_d} \F \M f\|_{L^2} &\lesssim \| \jbrak{x}^{\gamma_d} \M f \|_{L^2} \\
                    &\lesssim \| \jbrak{x}^{\gamma_d} f\|_{L^2} \\
                    &\lesssim t^{\eps^2} \| u\|_{\cX}.
                \end{align*}
                Thus the total contribution from $\|I_1(s)\|_{L^\infty}$ is 
                \begin{equation*}
                    \| I_1(s)\|_{L^\infty} \lesssim t^{-\frac{1}{50}+3\eps^2} \|u\|_{\cX}^3,
                \end{equation*}
                which is acceptable. 
                For $\|I_2(s)\|_{L^\infty}$, we have 
                \begin{align*}
                    \|I_2(s)\|_{L^\infty} &= \left\|(\K_t \ast |\F \M f|^2) (\F\M f) - (\K \ast |\hat{f}|^2)\hat{f} \,\right\|_{L^\infty} \\
                    &= \left\| (\K \ast |\hat{f}|^2)\hat{f} - (\K_t \ast |\F \M f|^2)\F \M f\right\|_{L^\infty} \\
                    &\leq \left\|(\K \ast |\hat{f}|^2) \hat{f} - (\K \ast |\F\M f|^2)\F \M f\right\|_{L^\infty} \\
                    &= \left\| (\K \ast |\F \M f|^2)\F \M f - (\K \ast |\hat{f}|^2)\hat{f} \, \right\|_{L^\infty}
                \end{align*}
                
                By writing 
                \[
                  (\K \ast |\F \M f|^2)\F\M f = (\K \ast |\F \M f|^2)(\F \M f - \hat{f} + \hat{f})
                \]
                and doing a bit of rearranging, we can rewrite 
                \begin{align*}
                    \|I_2(s)\|_{L^\infty} &\lesssim \left\| \K \ast |\F\M f|^2\right\|_{L^\infty} \| \F \M f - \hat{f}\|_{L^\infty} 
                    \\
                    &+ \left( \left\| \K \ast [((\F\M f)-\hat{f})\bbar{\hat{f}}] \, \right\|_{L^\infty} + \left\| \K \ast [(\bbar{\F\M f - \hat{f}})\hat{f}] \, \right\|_{L^\infty}\right)\|\hat{f}\|_{L^\infty}
                \end{align*}
                By the same trivial bound $\K(x) \leq |x|^{-1}$, we can bound $\|I_2(s)\|_{L^\infty}$ from above by 
                \begin{align*}
                    \|I_2(s)\|_{L^\infty}&\lesssim \left\| |\cdot|^{-1} \ast |\F\M f|^2\right\|_{L^\infty} \| \F \M f - \hat{f}\|_{L^\infty} 
                    \\
                    &\begin{multlined}
                        +  \left\| |\cdot|^{-1} \ast [((\F\M f)-\hat{f})\bbar{\hat{f}}] \, \right\|_{L^\infty}\|\hat{f}\|_{L^\infty} \\
                    + \left\| |\cdot|^{-1} \ast [(\bbar{\F\M f - \hat{f}}) \hat{f}] \, \right\|_{L^\infty}\|\hat{f}\|_{L^\infty}   
                    \end{multlined}
                     \\
                    &\lesssim t^{-\frac{1}{50}+ 3\eps^2} \|u\|_{\cX}^3
                \end{align*}
                where the final inequality comes from the estimates in \cite[equation (4.9)f]{hayashiAsymptoticsLargeTime1998a}. \par

                For $\|I_3(s)\|_{L^\infty}$ and $\|I_4(s)\|_{L^\infty}$ we can appeal directly to \cite[equations (3.16) and (3.17)]{hayashiAsymptoticsLargeTime1998a} to get
                \begin{align*}
                    \|I_3(s)\|_{L^\infty} &= \|\F(\M^{-1}-1)\F^{-1}|\F\M f|^{\frac{2}{d}}\F \M f\|_{L^\infty} \\
                    &\lesssim t^{-\frac{1}{50}+ (1+\frac{2}{d})\eps^2} \|u\|_{\cX}^{1+\frac{2}{d}}
                \end{align*}
                and 
                \begin{align*}
                    \|I_4(s)\|_{L^\infty} &= \left\|\F \M f|^{\frac{2}{d}}\F \M f - |\hat{f}|^{\frac{2}{d}} \hat{f} \,\right\|_{L^\infty} \\
                    &\lesssim t^{-\frac{1}{50} + (1+\frac{2}{d})\eps^2}\|u\|_{\cX}^{1+\frac{2}{d}}.
                \end{align*}
                which completes the proof. Indeed, we have 
                \begin{equation*}
                    \|g(t)\|_{L^\infty} \lesssim \|g(1)\|_{L^\infty} + \int_1^t s^{-1} \sum_{j = 1}^4 \|I_j(s)\|_{L^\infty} \ds.
                \end{equation*} 
               Recalling that $|g(t)| = |\hat{f}(t)|$ and that the local theory from above implies $\|g(1)\|_{L^\infty} \leq 2\eps$ by Sobolev embedding, we have 
                \begin{equation*}
                    \|u(t)\|_{\cD} \lesssim 2\eps + \int_1^t s^{-1-\frac{1}{50}+3\eps^2}\|u\|_{\cX}^3 + s^{-1-\frac{1}{50}+(1+\frac{2}{d})}\|u\|_{\cX}^{1+\frac{2}{d}} \ds, 
                \end{equation*}
                which is exactly what we set out to show. 
            \end{proof}
            By a continuity argument, the local theory from above, and the previous two lemmas, we are able to conclude the following
            \begin{corollary}
                For $\|u_0\|_{H^{\gamma_d, \gamma_d}} = \eps $ sufficiently small, there exists a unique forward-global solution $u \in C_t^0 H_x^{\gamma_d, \gamma_d}([0, \infty) \times \R^d)$ to \eqref{E:SBP} with $u(0, x) = u_0(x)$ satisfying 
                \begin{equation}
                    \|u(t)\|_{\cX} \lesssim \eps \qtq{for all} t \geq 1.
                \end{equation}
                In particular, we have the $L^\infty$ decay estimate
                \begin{equation}
                    \|u(t)\|_{L^\infty} \lesssim \eps(1+ |t|)^{-\frac{1}{2}}  \qtq{for all} t \geq 0. 
                \end{equation}
            \end{corollary}
        \section{Asymptotic Behavior of Solutions}\label{S:AsympBehavior}
            In this section we will write down the explicit asymptotic behavior for \eqref{E:SBP}. In particular, recall that in \Cref{P:higherdimDnorm}, we set 
            \[
              g(t)= B(t)\hat{f}(t) \qtq{with} B(t) = \exp\left(\frac{i}{2}\int_1^t \frac{1}{s} (\K \ast |\hat{f}|^2 - |\hat{f}|^{\frac{2}{d}}) \ds\right).
            \]
            The continuity argument from earlier now yields the estimate 
            \begin{equation}
                \|\partial_t g\|_{L^\infty} \lesssim \eps^{1+\frac{2}{d}}t^{-1-\frac{1}{50}+3\eps^2}
            \end{equation}
            In particular, by the fundamental theorem of calculus and the fact that $L^\infty$ is a Banach space, it follows that 
            \begin{equation}
                \|g(t) - \mathcal{W}_0 \|_{L^\infty} \lesssim \eps^{3}t^{-\frac{1}{50} + 3\eps^2}
            \end{equation}
            for some $\cW_0 \in L^\infty$. This tells us that $|\hat{f}| \to |\cW_0|$ in $L^\infty$. Plugging this into the definition of $B(t)$, we have 
            \begin{equation}
                B(t) = \exp\left(\frac{i}{2} (\K \ast |\cW_0|^2  - |\cW_0|^{\frac{2}{d}})+ \Phi(t)\right)
            \end{equation}
            where $\Phi(t)$ converges to a real-valued limit $\Phi_\infty$ in $L^\infty$, with rate $t^{-\frac{1}{25} + 2(1+\frac{2}{d})\eps^2}$. Note that we are required to choose the \textit{worst} rate of decay, which corresponds to the convolution nonlinearity. Defining \[\cW = \exp(-i\Phi_\infty)\cW_0,\] we find
            \begin{align*}
                \hat{f}(t) &= \exp\left(-\frac{i}{2}(\K \ast |W_0|^2 - |\cW_0|^{\frac{2}{d}}) \log(t)\right)\exp(-i\Phi_\infty)\cW_0 + \mathcal{O}(t^{-\frac{1}{50} + (1+\frac{2}{d})\eps^2})\\
                &= \exp\left(-\frac{i}{2} (\K \ast |\cW|^2 - |\cW|^{\frac{2}{d}}) \log(t)\right)\cW + \mathcal{O}(t^{-\frac{1}{50} + (1+\frac{2}{d})\eps^2}).
            \end{align*}
            Recalling that $u(t) = \M(t) \cD(t) \F \M(t) f(t) = \M(t)\cD(t) \hat{f} + \mathcal{O}(t^{-\frac{d}{2}-\frac{1}{50}+(1+\frac{2}{d})\eps^2})$, we conclude the desired result. 

            \nocite{*}
    \bibliography{bopp-podolsky}
    \bibliographystyle{bjoern_style}

\end{document}